\documentclass[oneside,english,11pt]{amsart}
\usepackage[latin1]{inputenc}

\usepackage{amsfonts}
\usepackage{amsthm}
\usepackage{amssymb}
\usepackage{graphicx}
\usepackage[all]{xy}
\usepackage{enumitem}
\usepackage{amsmath,bm}
\usepackage{graphpap}

\newtheorem{teo}{Theorem}[section]
\newtheorem{prop}[teo]{Proposition}
\newtheorem{ddef}[teo]{Definition}

\newtheorem{cor}[teo]{Corollary}

\newtheorem{lem}[teo]{Lemma}
\newtheorem*{cor*}{Corollary}
\newtheorem*{obss}{Remark}
\newtheorem*{teo*}{Theorem}

\newtheorem{maintheorem}{Theorem}

\newcommand{\co}{\mathbb{C}}

\newcommand{\qe}{\mathbb{Q}}
\newcommand{\ze}{\mathbb{Z}}

\newcommand{\cpt}[1]{\mathbb{C}P^{2}}

\newcommand{\pe}{\mathbb{P}}

\newcommand{\sing}{{\rm Sing}}

\newcommand{\C}{\mathbb{C}}

\newcommand{\DD}{{\mathcal{D}}}
\newcommand{\E}{{\mathcal{E}}}

\newcommand{\F}{\mathcal{F}}
\newcommand{\G}{\mathcal{G}}

\newcommand{\cl}[1]{\mbox{$\mathcal{#1}$}}

\newcommand{\bs}[1]{\mbox{$\boldsymbol{#1}$}}

\newcommand{\sep}{\mbox{\rm Sep}}
\newcommand{\e}{{\rm e}}

\newcommand{\dr}{\mbox{$\partial$}}

\newcommand{\gsv}{\mbox{\rm GSV}}
\newcommand{\bb}{\mbox{\rm BB}}
\newcommand{\cs}{\mbox{\rm CS}}

\begin{document}

\setcounter{section}{0}
\setcounter{teo}{0}
\setcounter{exe}{0}

\author{Gilberto Cuzzuol \&  Rog\'erio   Mol }
\title{ Second Type foliations of codimension one}


\begin{abstract}
In this article, for holomorphic foliations of codimension one at $(\mathbb{C}^3,0)$, we define the family of second type foliations. This  is formed by  foliations having,
in the reduction process by blow-up maps,  only well oriented singularities,   meaning that the reduction divisor does not contain weak separatrices of saddle-node singularities.
  We   prove that the reduction of
singularities of a non-dicritical foliation of second type coincides with the desingularization of its   set of separatrices.
\end{abstract}

\maketitle

\footnotetext[1]{ {\em 2010 Mathematics Subject Classification:}
 32S65, 37F75}
 \footnotetext[2]{{\em
Keywords.} Holomorphic foliations,  invariant varieties, equidesingularization.}
\footnotetext[3]{First author supported by a CAPES post-doctoral fellowship. Second author supported by Universal/CNPq. }

\section{Introduction}

An important category of problems in the theory of holomorphic foliations involves the study of equisingularity properties.
A question of this type  was addressed in \cite{camacho1984}, where the authors proved
that topological equivalent foliation inside the family of
 generalized curve   foliations are equisingular, meaning that their reductions of singularities
by blow-up maps are combinatorially equivalent. A germ of  foliation  at $(\C^{2},0)$ is said to be a \emph{generalized curve}    if there are no saddle-nodes in its reduction of singularities. In this family, separatrices --- formal invariant curves ---   are all  analytic and  carry an important
volume of topological information of the foliation itself. For instance,
  a generalized curve foliation and its set of separatrices have the same reduction of singularities,
meaning that a sequence of blow-ups that desingularizes all separatrices transforms the foliation into
 one having only simple singularities.
Generalized curve foliations are also
 characterized as those that
minimize Milnor numbers, once an equisingular set of separatrices is fixed.
For instance, when the set of  separatrices  of a foliation $\G$ is finite (the so-called \emph{non-dicritical case}) having $g=0$ as a reduced equation,
then $\G$ is a generalized curve if and only if $\mu_{0}(\G) = \mu_{0}(dg)$, where $\mu$ denotes the Milnor number.

A  quite  natural development is the extension  of the notion of generalized curve  for codimension
one foliations   in  higher dimension   spaces. In ambient dimension three,
 the existence of a
reduction of singularities \cite{cano1992, cano2004} has been the starting point in  \cite{mozo2009} for the de\-fi\-ni\-tion of \emph{generalized surface foliations}, comprising
 non-dicritical foliations (meaning that the divisor in the reduction process is invariant)     without  saddle-nodes in the reduction of singularities.
The main theorem in \cite{mozo2009} asserts that a generalized surface foliation and its set of separatrices have the
same reduction of singularities. We call the attention that, in the universe of dicritical codimension one foliations,
a result of this sort  does not make sense, since there are foliations --- for instance, the celebrated Jouanolou example \cite{jouanolou1979} --- having no separatrix at all.
The notion of generalized surface foliations reappears,  in   arbitrary ambient dimension,  in \cite{cano2015},  receiving the designation of \emph{complex hyperbolic}
foliations, a terminology we would rather adopt.
The  definition now  is  set in terms of two dimensional sections:
a germ of codimension one foliation $\F$ at $(\C^{n},0)$ is   complex hyperbolic if
and only if, for every analytic map $\phi:(\C^{2},0) \to (\C^{n},0)$ generically transversal to $\F$, the pull-back
foliation $\G = \phi^{*}\F$ is a generalized curve foliation.

Returning to dimension two, the conditions   defining
generalized curve foliations  can be slightly weakened,  delimiting
the larger family of  second type foliations. As introduced
 in \cite{mattei2004}, a germ of foliation  at $(\C^{2},0)$ is
of \emph{second type} or in the \emph{second class} if all singularities in its reduction of singularities
are well oriented (Definition \ref{def-well-oriented}). This means precisely  that  saddle-nodes are
admitted in the reduction process, provided they
lie in the regular part of the   divisor, with their weak separatrices transversal to it.
As well
as generalized curve foliations,   second type foliations
 and their sets of separatrices have equivalent reductions of singularities. However,    formal separatrices
 may exist. The minimization property now works for  the algebraic multiplicity.
For instance, if $g=0$ is a reduced equation for the set of separatrices  of a non-dicritical
foliation
$\G$, then $\G$ is second type if and only if $\nu_{0}(\G) = \nu_{0}(dg)$, where $\nu$ stands for
the algebraic multiplicity. These properties are the key ingredients used
in \cite{mol2016} in order to prove
that second type foliations equivalent by $C^{\infty}$ diffeomorphisms are equisingular.

The central objective of this article is to propose an extension of the concept of second type foliations
to foliations of codimension one. As in \cite{mozo2009},
we only work  in the three dimensional case, since we use a reduction of singularities in our definition. Evidently, for $n > 3$, as soon as the existence of a reduction process  for foliations
at $(\C^{n},0)$ is proved, the notion of second type foliations along with most of the results developed in this article can be properly adapted. Now, a germ of codimension one foliation at $(\C^{3},0)$
 is of second type if it admits a reduction of singularities in which all singularities are well oriented with respect
to the divisor (Definitions \ref{well-oriented-3d} and \ref{second-type-3d}).
It turns out that the definition does not depend on the reduction of singularities and that
it   may also be formulated in terms of two-dimensional sections (Proposition \ref{teo-second-type-3d}).
Employing  arguments similar to those   in \cite{mozo2009}, we can prove the following result:

\begin{maintheorem}
\label{theorem-main}
Let $\F$ be a non-dicritical foliation at $(\C^{3},0)$. Suppose that
 $\F$ is second type. Then $\F$ and its set of  separatrices have the same reduction
 of singularities.
\end{maintheorem}

This article has the following organization. In Section \ref{section-reduction} we briefly recall
the main aspects of the reduction of singularities for foliations, both  in dimension two
and  three. Next, in Section \ref{section-well-oriented}, we establish the notion of well
oriented singularities in dimension three, define
second type foliations  and prove 
that the   definition does not depend on the reduction of singularities and may also be formulated in terms of two dimensional sections. In Section \ref{section-desingularization}, we prove   Theorem \ref{theorem-main}.
Finally, in Section \ref{section-logarithmic}, we give a characterization  of logarithmic foliations
in the projective space $\pe^{3}$
in terms of the notion of second type foliation.

Part of this work was developed during a post-doctoral stage of the first author at Universidad
de Cantabria. The authors would like to express their gratitude to N. Corral.

\section{Simple singularities and reduction of singularities}
\label{section-reduction}
We start by establishing a pattern for the notation to be followed. We use $(u,v)$ for analytic or formal local coordinates at $(\C^{2},0)$
and $(x,y,z)$  at $(\C^{3},0)$. The term foliation is used as a short for   \emph{singular holomorphic foliation of codimension one}. Foliations and normal crossings divisors are denoted, with variations, respectively
by $\G$ and $\E$, in dimension two, and by $\F$ and $\DD$,  in dimension three.
 The  number of local branches of a divisor at a point $p$ is denoted   by
 $\e_{p}(\E) $ or $\e_{p}(\DD)$.   Abusing terminology and notation, an empty  local normal crossings
divisors at a point $p$ is
represented   by the one point set $\{p\}$. For $n=2$ or $3$, as usual,   $\cl{O}_n$ and $\hat{\cl{O}}_n$   denote, respectively, the local rings of analytic functions and of formal functions  $n$ variables.

\subsection{In dimension two}
\label{subsection-tau2}
A germ of holomorphic foliation $\G$ at $(\C^{2},0)$ is defined, in analytic coordinates $(u,v)$,
by an analytic
$1-$form
\begin{equation}
\label{oneform-tau2}
\eta = A(u,v) d u + B(u,v) d v,
\end{equation}
where  $A, B  \in \cl{O}_{2}$ are relatively prime.
A  \textit{separatrix} for  $\cl{G}$ is an invariant formal irreducible curve,
corresponding to an irreducible formal function
$f\in \hat{\cl{O}}_{2}$   satisfying
$$
\eta \wedge df=(fh) du\wedge dv
$$
for some
$h\in \hat{\cl{O}}_{2}$.
The separatrix is said to be {\em analytic}   if
we can take $f, h \in \cl{O}_{2}$.  The set of separatrices of $\G$ is denoted by $\sep_0(\G)$.
Following the usual notation, its \emph{algebraic multiplicity} is $\nu_{0}(\G) = \min\{\nu_{0}(A),\nu_{0}(B)\}$ and
its \emph{Milnor number} is $\mu_{0}(\G) = \dim_{\C} \cl{O}_{2}/(A,B)$.

The foliation $\cl{G}$ is \emph{simple}     if the
 linear part of  $\bs{v} = B(u,v) \dr / \dr u - A(u,v)\dr / \dr v$, vector field dual to $\eta$, is  non-nilpotent and has eigenvalues with quotient outside  $\qe_{>0}$. Simple foliations   admit the following   formal normal forms:
 \begin{itemize}[leftmargin=*]
\item[]
\begin{equation}  \label{a}
\tag{$a$} \eta = uv \left( \lambda_{1} \frac{du}{u} + \lambda_{2} \frac{dv}{v} \right),
\end{equation}
where $\lambda_{1},  \lambda_{2} \in \C^{*}$ and $m_{1} \lambda_1 + m_{2} \lambda_2 \neq 0$
 for every $m_{1}, m_{2} \in \ze_{\geq 0}$ with at least one of them non-zero;
\item[]
\begin{equation}  \label{b1}
\tag{$b1$} \eta = uv \left(  \frac{du}{u} + \varphi(u) \frac{dv}{v} \right),
\end{equation}
where  $\varphi \in \hat{\cl{O}}_{1}$ is a non-unity;
\item[]
\begin{equation} \label{b2}
\tag{$b2$} \eta = uv \left( p_{1} \frac{du}{u} + p_{2} \frac{dv}{v}
+ \varphi(u^{p_{1}} v^{p_{2}}) \frac{dv}{v} \right),
\end{equation}
where $p_{1}, p_{2} \in \ze_{>0}$ and   $\varphi \in \hat{\cl{O}}_{1}$   is a non-unity.
\end{itemize}
Models \eqref{a} and \eqref{b2}  are called {\em non-degenerate} or {\em complex hyperbolic} simple singularities.
Model \eqref{b1}, corresponding to the existence of a zero eigenvalue,  is a \emph{saddle-node} singularity.
For all simple foliations, the separatrix set   $\sep_0(\G)$ is formed by two
transversal   branches, given by $\{u=0\}$ and $\{v=0\}$. In the non-degenerate case, both are analytic.
For a saddle-node, the separatrix corresponding to $\{u=0\}$, associated to the eigenspace of the non-zero eigenvalue,
is analytic and is called {\em strong}. On its turn,   $\{v=0\}$ defines a possibly formal separatrix, called {\em weak}.

Let  $\E$  be normal crossings divisor  at $(\C^{2},0)$
such that either $\e_{0}(\E) = 1$ or $2$. We   say that
 $\G$ is:
 \begin{itemize}
\item $\E$-\emph{regular},    if there
are local analytic coordinates $(u,v)$ at the origin such that $\E \subset \{uv=0\}$
and   $\G: du = 0$;
\item $\E$-\emph{simple},    if $\G$ has a simple singularity
and   $\E \subset \sep_{0}(\G)$.
\end{itemize}
In this context, we say that $\G$ is \emph{adapted} to the divisor $\E$.
The theorem of  reduction of singularities for foliations in dimension two can be stated in the following form:
\begin{teo*}[Reduction of singularities, dimension two  \cite{seidenberg1968,camacho1984}]
 \label{teo-reduction-C2}
 Let $\G$ be a foliation at $(\C^{2},0)$. Then
there exists a proper analytic map, formed by a composition of   blow-up maps,  $\sigma: (\tilde{N},\E) \to
(\co^2,0)$, where $\E = \sigma^{-1}(0)$ is a normal crossings divisor and $\tilde{N}$ is a germ of complex surface around
$\E$, such that all points of   $\E$ are either $\E$-regular or $\E$-simple singularities
 for the transformed foliation $\tilde{\G} = \sigma^{*}\G$.
\end{teo*}
The map $\sigma$ above is called  {\em reduction of singularities} or   {\em desingularization}  for   $\G$. The concept of simple singularity   well oriented with respect to a divisor was established in \cite{paul1999}. Here is the precise definition:

\begin{ddef}
\label{def-well-oriented}
{\rm
Let $\E$ be a normal crossings divisor and $\G$ be an $\E$-simple germ of foliation  at $(\C^{2},0)$.
We say that $\G$   is
$\E$-\emph{well oriented} in
one of the following cases:
\begin{enumerate}[label={(\roman*)}]
\item $\G$ is a non-degenerate singularity;
\item $\G$ is a saddle-node singularity whose weak separatrix is not contained in $\E$.
\end{enumerate}
Otherwise, we say that $\G$ is an  $\E$-\emph{tangent saddle-node}.
}\end{ddef}
The notion of $\E$-well oriented simple singularity is invariant by blow-ups in the following sense:
if $\sigma$ is a blow-up map at $0 \in \C^{2}$ with exceptional
divisor $E = \sigma^{-1}(0)$,   then
$\tilde{\G} = \sigma^{*} \G$
has two simple singularities, corresponding to
the  two points of intersection, say $p_{1}$ and $p_{2}$, between $E$ and the transforms of the  branches of $\sep_{0}(\G)$. If $\G$ is $\E$-well oriented, then $\tilde{\G}$ is well oriented with respect
to $\tilde{\E} = E \cup \sigma^{*}\E$ at $p_{1}$ and $p_{2}$. On the other hand,
  if $\G$ is an  $\E$-tangent saddle-node  whose weak separatrix lies in a component
 $E_{1} \subset  \E$, then $\tilde{\G}$ has a saddle-node singularity  at $p_{1} = E \cap \sigma^{*}E_{1}$
 whose weak separatrix is contained in $\sigma^{*}E_{1}$, being
 an $\tilde{\E}$-tangent saddle-node, whereas at the other singular point $p_{2} \in E$, we have that $\tilde{\G}$ is simple non-degenerate  and thus $\tilde{\E}$-well oriented.

The concept of well oriented singularities is the basis for the following definition \cite{mattei2004}:

\begin{ddef}
{\rm
A germ of foliation $\G$   at $(\C^{2},0)$ is said to be of \emph{second type} if, given a reduction
process $\sigma: (\tilde{N},\E) \to
(\co^2,0)$,   all singularities of the transformed foliation $\tilde{\G} = \sigma^{*} \G$ are
$\E$-well oriented.
}\end{ddef}
Clearly, the definition does not depend on the reduction of singularities.
Now, for a fixed
 local normal crossing divisors $\E$ at $(\C^{2},0)$, an
 $\E$-\emph{reduction of singularities} is a map $\sigma: (\tilde{N},\E^{\#})\to
(\co^2,0)$, composition of a finite number of blow-ups, such that all points in
$\tilde{\E} = \sigma^{-1}(\E) = \E^{\#}\cup \sigma^{*} \E$ are either $\tilde{\E}$-simple
or $\tilde{\E}$-regular for $\tilde{\G}= \sigma^{*}\G$.
In this case,
we say that $\G$ is
$\E$-\emph{second type}
if all singularities of  $\tilde{\G}$ are  $\tilde{\E}$-well oriented.
 For instance, if $\G$ is a saddle-node singularity, then it
is a second type foliation. However, setting $\E = \sep_{0}(\G)$ or $\E = \{\text{weak separatrix}\}$,
then  $\G$ is not $\E$-second type.

Let $\G$ be a non-dicritical foliation at $(\C^{2},0)$ whose
set of separatrices has $g=0$ as a reduced equation, where $g \in \hat{\cl{O}}_{2}$.
Then $\nu_{0}(\G) \geq \nu_{0}(dg)= \nu_{0}(\sep_{0}(\G)) -1$ and  the equality holds if and only if $\G$ is second type  \cite{mattei2004}.
This property of minimization of the algebraic multiplicity also holds in the dicritical case
and a formulation for it can be seen in \cite{genzmerl2017}.

A second type foliation and its set of separatrices have equivalent reductions of
singularities. This is a straight consequence of the following lemma, which is
a restatement of Lemma 1 in \cite{camacho1984}:
\begin{lem}
\label{lemma-simple2} Let $\G$ be a germ of foliation at $(\C^{2},0)$. Suppose that
\begin{enumerate}[label={(\roman*)}]
\item $\sep_{0}(\G)$ has exactly two transversal branches;
\item $\G$ is second type.
\end{enumerate}
Then $\G$ is simple.
\end{lem}
\begin{proof} The proof is essentially that   of \cite{camacho1984}.
We take formal coordinates $(u,v)$ such that $\sep_{0}(\G) = \{uv=0\}$, implying that
$\G$ is given by a   $1-$form of the kind $\eta = v \tilde{a} du + u \tilde{b} dv$,
where $\tilde{a}, \tilde{b} \in \hat{\cl{O}}_{2}$.
Since $\G$ is second type, $\nu_{0}(\G) = \nu_{0}(\sep_{0}(\G)) - 1 = 1$.
Thus, the linear part of $\eta$ is $ \lambda_{1} v du + \lambda_{2} u dv$, with
$\lambda_{1},  \lambda_{2}$ not both zero. If both are non-zero, it suffices
to see that we cannot have  $\lambda_{1}/  \lambda_{2} \in \qe_{<0}$. Actually,
if this happens, either  $\sep_{0}(\G)$ has a unique branch ($\G$ non-linearizable with either $\lambda_{1}/  \lambda_{2} \in \ze_{<0}$
or $\lambda_{2}/  \lambda_{1} \in \ze_{<0}$) or  $\sep_{0}(\G)$ has infinitely many
branches (all other cases).
\end{proof}

\subsection{In dimension three}
\label{subsection-reduction-3d}
A germ of codimension one holomorphic foliation $\F$ at $(\C^{3},0)$ is defined, in analytic coordinates $(x,y,z)$,
by an analytic
$1-$form
\begin{equation}
\label{oneform-tau3}
\omega = A(x,y,z) d x + B(x,y,z) d y + C(x,y,z)d z,
\end{equation}
where  $A, B, C  \in \cl{O}_{3}$ are without common factors, satisfying the integrability condition
$\omega \wedge d \omega = 0$.
A
\textit{separatrix} for  $\cl{F}$ is an invariant formal irreducible surface, that
is, an object given by an irreducible formal function
$f\in \hat{\cl{O}}_{3}$   such that
$ \omega\wedge df=(fh) \theta$
for some
$h\in \hat{\cl{O}}_{3}$ and some formal $2-$form $\theta$. We also have the obvious notion of {\em analytic} separatrix.
The set of separatrices of $\F$ is again denoted by $\sep_0(\F)$.

The \emph{dimensional type} of a foliation $\F$, denoted by $\tau_{0}(\F)$ or shortly by $\tau$, is the
smallest number of variables needed to express its defining equation
in some system of analytic coordinates. Thus, $\tau_{0}(\F) = 2$
if and only if there are analytic coordinates $(x,y,z)$ under which $\F$  is an analytic cylinder  over a singular     foliation in the
  coordinates $(x,y)$. Note that $\tau_{0}(\F) = 1$ if and only if $\F$  regular.

Let us list the simple formal models for singularities of a foliation $\F$  at $(\C^{3},0)$. If
$\tau_{0}(\F) = 2$,  they have already been listed in   Subsection \ref{subsection-tau2}.
If $\tau_{0}(\F) = 3$, we say that $\F$ is \emph{simple} if there are formal coordinates
 $(x,y,z)$ in which $\F$ is expressed in one of the following models:

\begin{itemize}[leftmargin=*]
\item[]
\begin{equation}  \label{A}
\tag{$A$} \omega=xyz\left( \lambda_1\dfrac{dx}{x} + \lambda_2 \dfrac{dy}{y} + \lambda_3\dfrac{dz}{z}\right),
\end{equation}
where $\lambda_1, \lambda_2, \lambda_3 \in \C^{*}$ are such that $m_1\lambda_1+m_2\lambda_2+m_3\lambda_3\neq 0$ whenever $m_1,m_2,m_3\in\mathbb{Z}_{\geq 0}$ with at least one of them non-zero;

\item[]
\begin{equation}  \label{B1}
\tag{$B1$}  \omega= xyz\left( p_1\dfrac{dx}{x}+\varphi(x^{p_1})\left( \lambda_2\dfrac{dy}{y}+\lambda_3\dfrac{dz}{z}\right)\right),
\end{equation}
where $p_1\in\mathbb{Z}_{>0}$,  $\varphi \in \hat{\cl{O}}_{1}$ is a non-unity and $\lambda_2, \lambda_3 \in \C^{*}$  satisfy $m_2\lambda_2+m_3\lambda_3\neq 0$ whenever $m_2,m_3\in\mathbb{Z}_{\geq 0}$ with at least one of them non-zero;

\item[]
\begin{equation}  \label{B2}
\tag{$B2$}
\omega= xyz\left( p_1\dfrac{dx}{x}+p_2\dfrac{dy}{y}+\varphi(x^{p_1}y^{p_2})\left( \lambda_2\dfrac{dy}{y}+\lambda_3\dfrac{dz}{z}\right)\right),
\end{equation}
where $p_1,p_2\in\mathbb{Z}_{>0}$,    $\varphi \in \hat{\cl{O}}_{1}$ is a non-unity and $\lambda_2, \lambda_3 \in \C^{*}$  satisfy $m_2\lambda_2+m_3\lambda_3\neq 0$ whenever $m_2,m_3\in\mathbb{Z}_{\geq 0}$ with at least one of them non-zero;

\item[]
\begin{equation}  \label{B3}
\tag{$B3$}
\omega=xyz\left( p_1\dfrac{dx}{x}+p_2\dfrac{dy}{y}+p_3\dfrac{dz}{z}+\varphi(x^{p_1}y^{p_2} z^{p_{3}}) \left(\lambda_2\dfrac{dy}{y}+\lambda_3\dfrac{dz}{z}\right)\right),
\end{equation}
where $p_1,p_2,p_3\in \mathbb{Z}_{>0}$, $\varphi \in \hat{\cl{O}}_{1}$ is a non-unity and $\lambda_2, \lambda_3 \in \C^{*}$  satisfy $m_2\lambda_2+m_3\lambda_3\neq 0$ whenever $m_2,m_3\in\mathbb{Z}_{\geq 0}$ with at least one of them non-zero.

\end{itemize}

Foliations of formal models \eqref{A}   or \eqref{B3}  are said to be simple  \emph{complex hyperbolic},
whereas those corresponding to models \eqref{B1} or \eqref{B2} are said to be \emph{saddle-nodes}.
In all of them, the singular set is formed by pairwise transversal   analytic curves corresponding to the coordinate axis.  Outside the origin, the
singularities are of dimensional type two. The set of separatrices is precisely the union of the
three coordinate planes.

For the complex hyperbolic models, \eqref{A} and \eqref{B3},
the transversal type along all coordinate axes is complex hyperbolic and,
furthermore, all separatrices are analytic.
Concerning simple saddle-node models, we can say the following:

\begin{itemize}[leftmargin=*]
\item In model \eqref{B1}, the transversal model along the $x$-axis is complex hyperbolic.
The $y$-axis and the $z$-axis have
  transversal models of  saddle-node type. For a transversal section  $y = c$, $c \neq 0$, the weak separatrix is contained in
$z = 0$ and the strong separatrix is in $x  = 0$.
On the other hand, for a transversal section  $z = c$, $c \neq 0$, the weak separatrix is contained in
$y = 0$ and the strong separatrix is in $x = 0$.

\item In model \eqref{B2},  the $z$-axis has  complex hyperbolic    transversal model, whereas
both  the $x$-axis and the $y$-axis have
transversal models  of saddle-node type. For a transversal section  $x = c$, $c \neq 0$, the weak separatrix is contained in
$z = 0$ and the strong separatrix is in $y  = 0$.
For a transversal section  $y = c$, $c \neq 0$, the weak separatrix is contained in
$z = 0$ and the strong separatrix is in $x  = 0$.
\end{itemize}

\vspace{1cm}
\begin{figure}[h]
\begin{center}
\includegraphics[height=4cm,width=12cm]{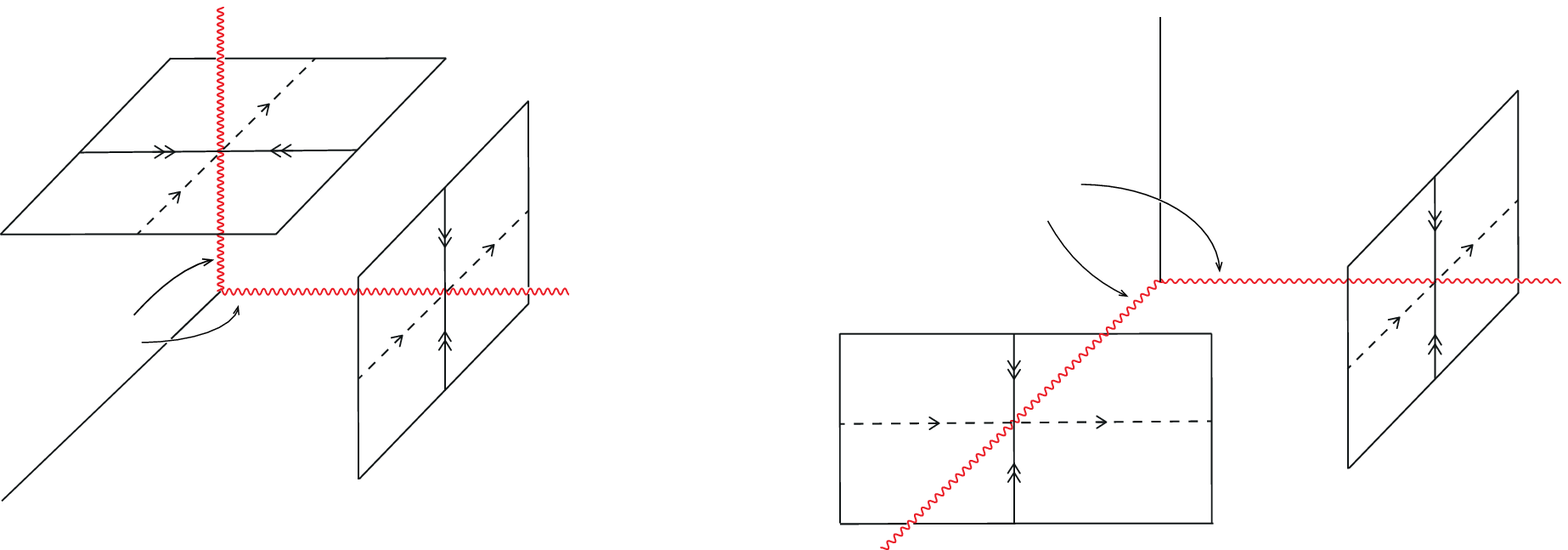}
\begin{picture}(0,0)
\put(-350,130){\footnotesize {Type \eqref{B1}}}
\put(-140,130){\footnotesize {Type \eqref{B2}}}
\put(-167,74){\footnotesize{ \textit{saddle-nodes,}}}
\put(-153,65){\footnotesize{ $\tau=2$}}
\put(-372,45){\footnotesize{ \textit{saddle-nodes,}}}
\put(-357,36){\footnotesize{ $\tau=2$}}
\put(-343,4){\footnotesize {$x$}}
\put(-151,-6){\footnotesize {$x$}}
\put(-223,47.5){\footnotesize{ $y$}}
\put(-4,50){\footnotesize {$y$}}
\put(-93,112){\footnotesize{ $z$}}
\put(-295,114.5){\footnotesize {$z$}}
\put(-370,-18){\footnotesize{weak separatrices: $\{y=0\}$ and $\{z=0\}$}}
\put(-124,-18){\footnotesize{weak separatrix: $\{z=0\}$}}
\end{picture}
\end{center}
\vspace{1cm}
\end{figure}

 This discussion founds the following definition:
\begin{ddef}{\rm  Let $\F$ be a simple foliation at $(\C^{3},0)$ of saddle-node type. We say that a germ of separatrix of $\F$   is
\emph{weak}   if
it contains a component $\Lambda$ of $\sing_{0}(\F)$ with saddle-node transversal type such that,
 for every $p \in \Lambda \setminus \{0\}$ and every two-dimensional section transversal to
$\Lambda$ at $p$, $\Lambda \cap \Sigma$ is the weak separatrix of $\F|_{\Sigma}$.
}\end{ddef}

According to this definition, a saddle-node singularity of model  \eqref{B1}
has two  weak separatrices, corresponding to
the planes $y = 0$ and $z = 0$. On its turn,
  a saddle-node   of model  \eqref{B2} has a unique  weak separatrix, corresponding to $z = 0$.
In dimensional type two, the notion of weak separatrix is clear.

At $(\C^{3},0)$, let $\F$ be a foliation defined by an integrable $1-$form $\omega$ and  $\Gamma$ be a germ of formal curve  with Puiseux parametrization
$\gamma(t)$.  We say that $\Gamma$ is \emph{invariant} by $\F$ if $\gamma^{*} \omega \equiv 0$. This
notion evidently does not depend on the choices made. For later use, we make explicit the following  simple fact:
\begin{lem}
\label{invariant-curve}
Let $\F$ be a simple foliation at $(\C^{3},0)$. Then all formal curves
invariant by $\F$ are contained in $\sep_{0}(\F)$.
\end{lem}
\begin{proof} The result is clear if $\tau_{0}(\F) = 2$.
If $\tau_{0}(\F) = 3$, it follows from straight calculations
with each simple formal model.
If the $\F$-invariant curve lies outside the coordinate planes, then
it has a parametrization $\gamma(t) =
(\gamma_{1},\gamma_{2}, \gamma_{3}) = (t^{n_{1}}\tilde{\gamma}_{1},t^{n_{2}} \tilde{\gamma}_{2}, t^{n_{3}}\tilde{\gamma}_{3})$,
where $n_{1}, n_{2}, n_{3} \in \ze_{>0}$ and $\tilde{\gamma}_{1}, \tilde{\gamma}_{2}, \tilde{\gamma}_{3} \in \hat{\cl{O}}_{1}$ are unities.
For model \eqref{B1}, for instance,
\[ \gamma^{*} \omega =   \gamma_{1} \gamma_{2} \gamma_{3}
\left[ p_{1} n_{1} \frac{dt}{t} + p_{1}   \frac{d\tilde{\gamma}_{1} }{\tilde{\gamma}_{1}}
+ \varphi( \gamma_{1}^{p_{1}})
\left( (\lambda_{2}n_{2} +  \lambda_{3}n_{3}) \frac{dt}{t} +
\lambda_{2} \frac{d\tilde{\gamma}_{2} }{\tilde{\gamma}_{2}} +
\lambda_{3}\frac{d\tilde{\gamma}_{3} }{\tilde{\gamma}_{3}} \right) \right]
\]
is identically zero, so the residue $p_{1}n_{1}$ of the formal meromorphic $1-$form
inside brackets is also zero, which is absurd.
The argument is the same for the other models.
\end{proof}

 Suppose that $\F$ is a foliation having
either a simple singularity or a non-singular point   at $0 \in \C^{3}$.
Let $\DD$ be a normal crossings divisor.
We decompose $\DD = \DD' \cup  \DD^{*}$, where $\DD'$ assembles
the $\F$-invariant components, called \emph{non-dicritical},  and $\DD^{*}$ the non-invariant ones, called \emph{dicritical}.
We say that
$\F$ is  \emph{adapted} to $\DD$ if:
\begin{itemize}
\item  $\tau_{0}(\F) -1 \leq \e_{0}(\DD') \leq \tau_{0}(\F)$;
\item  $\e_{0}(\DD^{*}) \leq 3 - \tau_{0}(\F)$.
\end{itemize}
In analogy with the two dimensional case, a foliation $\F$ adapted to a divisor $\DD$ can be:
\begin{itemize}
\item \emph{$\F$-simple}, if either $\tau_{0}(\F) = 3$ or $\tau_{0}(\F) = 2$ and there are local coordinates
 $(x,y,z)$ such that $\F$ is expressed in the variables $(x,y)$ and $\DD^{*} \subset \{z=0\}$;
\item \emph{$\F$-regular},  if   there are local coordinates
 $(x,y,z)$ such that $\tilde{\F}$ is given by $dx=0$   and $\DD^{*} \subset \{yz=0\}$;
\end{itemize}

In dimension three, the existence of a reduction of singularities is stated in the following way:
\begin{teo*}[Reduction of singularities, dimension three  \cite{cano1992,cano2004}]
 \label{teo-reduction}
Let $\F$ be a holomorphic singular foliation of codimension one at $(\C^{3},0)$. Then
 there is an analytic map $\pi: (\tilde{M}, \DD) \to (\C^{3}, \sing(\F))$ formed by the
composition of a finite sequence of blow-up maps  with non-singular centers such
that:
\begin{enumerate}
\item at each intermediate step, the blow-up  center is invariant for the corresponding transformed foliation and has normal crossings with the corresponding blow-up divisor;
\item  the transformed foliation $\tilde{\F} = \pi^{*} \F$ is such that
all points in $\DD$ are either     $\DD$-simple or  $\DD$-regular.
\end{enumerate}
\end{teo*}
Following the local picture, we decompose that reduction divisor in non-dicritical and dicritical components, getting  $\DD = \DD' \cup  \DD^{*}$.
The foliation is said to be \emph{non-dicritical} if $\DD = \DD'$, that is, if the reduction divisor
is $\tilde{\F}$-invariant. Finally, singularities   $p \in \sing(\tilde{\F})$  are classified  in two groups:
\begin{itemize}
\item \emph{simple corners}, if $\e_{p}(D') =  \tau_{p}(\tilde{\F})$ (local separatrices contained
in $\DD$);
\item  \emph{trace singularities}, if $\e_{p}(D') =  \tau_{p}(\tilde{\F}) -1 $  (there is one local separatrix outside $\DD$).
\end{itemize}


\section{Second type foliations in dimension three}
\label{section-well-oriented}

The definitions of well oriented singularities and tangent saddle-nodes for   foliations
in dimension three is   analogous to those in dimension two:

\begin{ddef}
\label{well-oriented-3d}
 {\rm Let $\DD$ be a normal crossings divisor and  $\F$ be a germ of $\DD$-simple foliation  at $(\mathbb{C}^3,0)$.
The foliation $\F$   is said to be
$\DD$-\emph{well oriented} in
one of the following cases:
\begin{enumerate}[label={(\roman*)}]
\item $\F$ is simple complex hyperbolic singularity (types \eqref{A} or \eqref{B3});
\item $\F$ is a saddle-node singularity having no weak separatrix   contained in $\DD$.
\end{enumerate}
Otherwise, we say that $\F$ is a   $\DD$-\emph{tangent saddle-node}.
}\end{ddef}

Let us identify the tangent saddle-nodes. We have a $\DD$-simple foliation $\F$, for a normal crossings divisor $\DD$ that has a decomposition
 $\DD = \DD' \cup \DD^{*}$  in non-dicritical and dicritical components.
Suppose that     $\tau_{0}(\F) = 3$,  so that $\sep_{0}(\F)$ has three local branches
corres\-pon\-ding, in the formal models in Subsection \ref{subsection-reduction-3d}, to the three coordinate planes.  In this case,
$\e_{0}(\DD) = \e_{0}(\DD') = 2$ or  $3$. Thus, we have the following possibilities for a
$\DD$-tangent saddle-node:
\begin{enumerate}[label=(\roman*)]
\item  $\e_{0}(\DD') = 2$ and $\F$ is either a saddle-node of model \eqref{B1}
or a saddle-node of model \eqref{B2}  having its weak separatrix in $\DD$;
\item $\e_{0}(\DD') = 3$ and $\F$ is a saddle-node, models \eqref{B1}   or \eqref{B2}.
\end{enumerate}
When $\tau_{0}(\F) = 2$, the picture is essentially two-dimensional: $\sep_{0}(\F)$ has two local branches
crossing normally,   and $\e_{0}(\DD') = 1$ or  $2$.
A  $\DD$-tangent saddle-node corresponds to  the following cases:
\begin{enumerate}[label=(\roman*)]
\item  $\e_{0}(\DD') = 1$ and $\F$ is a saddle-node having its weak separatrix in $\DD$;
\item $\e_{0}(\DD') = 2$ and $\F$ is a saddle-node.
\end{enumerate}

We need the following definition:
\begin{ddef}
{\rm Let $\F$ be a germ of foliation at $(\C^{3},0)$ defined by $\omega = 0$  and let $\E$
be a normal crossings divisor   $(\C^{2},0)$.
An analytic map  $\phi:(\C^{2},0) \to (\C^{3},0)$ is \emph{generically
transversal} to $\F$ \emph{outside} $\E$ if  $\sing(\phi^{*} \omega) \subset \E$. When $\E = \{0\}$ we
  simply say that $\phi$     is  \emph{generically
transversal} to $\F$.
}\end{ddef}

If $\F$ is a foliation at $(\C^{3},0)$ then  there are embeddings $\phi:(\C^{2},0) \to (\C^{3},0)$
  generically transversal to $\F$   satisfying   $\nu_{0}(\G) = \nu_{0}(\F)$, where
$\G = \phi^{*} \F$ \cite{mattei1980}. This fact is needed in the proof of Lemma \ref{lemma-simple3} below, a three-dimensional analogue of Lemma \ref{lemma-simple2}.
For our purposes, a relevant property of
generically transversal analytic maps is that they   preserve   well orientation of
 simple singularities:
\begin{lem}
\label{lemma-fund}
Let $\DD$ be a normal crossings divisor and $\F$ be a  $\DD$-well oriented germ of $\DD$-simple foliation at $(\C^{3},0)$.
 Let $\E$ be a normal crossing divisor at $(\C^{2},0)$ and
suppose that
 $\phi: (\C^{2},\E) \to  (\C^{3},\DD)$ is  an analytic
map generically transversal to $\F$ outside $\E$.
If $\G = \phi^{*} \F$ is $\E$-simple, then
$\G$ is  $\E$-well oriented.
\end{lem}
\begin{proof}
If $\F$ is a simple complex hyperbolic, then $\G = \phi^{*} \F$ is   simple non-degenerate      by \cite[Lemma 4.3]{cano2015}. Thus, we can restrict ourselves to the case where $\F$ is a $\DD$-well oriented saddle-node.
We fix  $(u,v)$  normalizing formal coordinates for $\G$, so that $\E \subset \sep(\G) = \{uv = 0\}$.
 We also fix normalizing formal coordinates $(x,y,z)$
for $\F$. Thus, all our objects become formal, including the map
 $\phi: (\C^{2},0) \to (\C^{3},0)$, that can be
  written as  $\phi(u,v) = (\phi_{1},\phi_{2},\phi_{3}) = (x,y,z)$.
 We    factorize the maximal powers of $u$ and $v$ from each  $\phi_{i}$,
 getting
\begin{equation}
\label{phi-factorization}
   \phi_{i}(u,v) = u^{r_{i}}v^{s_{i}} \tilde{\phi}_{i}(u,v), \ \ \ \text{for}\ i=1,2,3.
\end{equation}
Since $\E$ is $\G$-invariant, by Lemma \ref{invariant-curve}, it is mapped into the invariant part of $\DD$.
Thus, the eventual dicritical components  do not intervene in our analysis and we make
a simplification by
supposing that that $\DD$ is $\F$-invariant.
 We have two situations, depending on the dimensional type of $\F$:
\smallskip
\par \noindent \underline{Case I}:   $\tau_{0}(\F) = 2$. We take for $\F$, in coordinates $(x,y,x)$,
 the formal normal form \eqref{b1}:
\[ \omega = xy \left(  \frac{dx}{x} + \varphi(x) \frac{dy}{y} \right) = xy \tilde{\omega}. \]
In this way, $\DD = \{x=0\}$ and the weak separatrix is $\{y = 0\}$.
No curve outside $\{uv = 0\}$ can be $\G$-invariant. Hence both $\tilde{\phi}_{1}$ and  $\tilde{\phi}_{2}$ are unities.
We have
 \begin{equation}
 \label{pull-back-tau2}
 \phi^{*} \tilde{\omega}  =  r_{1} \frac{du}{u} + s_{1} \frac{dv}{v} +  \frac{d \tilde{\phi}_{1}}{\tilde{\phi}_{1}}
 + \varphi(u^{r_{1} }v^{s_{1} }  \tilde{\phi}_{1}) \left(    r_{2} \frac{du}{u} + s_{2} \frac{dv}{v}  + \frac{d \tilde{\phi}_{2}}{\tilde{\phi}_{2}} \right).
\end{equation}

Our analysis splits into two possibilities:

\smallskip
\par \noindent \underline{Subcase I.1}:\ $\E$ has two branches, $\E =  \{uv = 0\}$.
Since $\phi$ takes $\E$ into $\DD$, we must have $r_{1} > 0$ and $s_{1} > 0$.
By cancelling poles in \eqref{pull-back-tau2}, we find
\begin{equation}
\label{CHfinal}
 uv \phi^{*} \tilde{\omega} = r_{1} v du + s_{1} u dv + u v  \, \theta ,
 \end{equation}
for some formal $1-$form $\theta$. Thus, $\G$ is simple non-degenerate (of formal type \eqref{b2}).

\smallskip
\par \noindent \underline{Subcase I.2}:\ $\E$ has a single branch. Let us suppose $\E = \{v = 0\}$.
Since $\phi$   takes $\E$ into $\DD$, we have that $s_{1}>0$.
By Lemma \ref{invariant-curve}, the image of     $\Gamma = \{u = 0\}$   must be contained   in some $\F$-invariant plane
--- either in $\DD = \{x=0\}$ or in
$\{y=0\}$ ---
not in both, for $\phi$ is   generically transversal   outside $\E$. According to this, we have:
\begin{itemize}[leftmargin=*]
\item If $\phi(\Gamma) \subset \DD$, then $r_{1} > 0$ and $r_{2} = 0$.
We recover the situation of the previous subcase:
  $\G$ is induced by a $1-$form of the type \eqref{CHfinal}, having a simple non-degenerate singularity
at the origin.
\item
If $\phi(\Gamma) \subset \{y=0\}$, then $r_{1} = 0$ and $r_{2} > 0$.
By cancelling poles in \eqref{pull-back-tau2}, we get
\[uv \phi^{*} \tilde{\omega}= s_{1} u dv +
 r_{2}  v   \varphi( v^{s_{1}}\tilde{\phi}_{1} )    du +
 u v  \, \theta ,\]
 for some formal $1-$form $\theta$. Thus, $\G$ has a saddle-node singularity
 at the origin whose weak separatrix is $\{u=0\}$. It is therefore    $\E$-well oriented.
\end{itemize}

\smallskip
\par \noindent \underline{Case II}:   $\tau_{0}(\F) = 3$. We can suppose that $\F$ is of type \eqref{B2} and that $\DD = \{xy=0\}$.
The weak separatrix is $\{z=0\}$.
This time,   $\tilde{\phi}_{1}$, $\tilde{\phi}_{2}$ and $\tilde{\phi}_{3}$ in \eqref{phi-factorization} are unities.
Writing $\omega  = xyz \tilde{\omega}$, we find
\begin{eqnarray*}
\phi^{*} \tilde{\omega} & = & p_{1} \left( r_{1} \frac{du}{u} + s_{1} \frac{dv}{v} + \frac{d \tilde{\phi}_{1}}{\tilde{\phi}_{1}} \right)
+ p_{2}\left( r_{2} \frac{du}{u} + s_{2} \frac{dv}{v} + \frac{d \tilde{\phi}_{2}}{\tilde{\phi}_{2}} \right) + \medskip \\
   &  &  \varphi(\phi_{1}^{p_{1}} \phi_{2}^{p_{2}})
\left(
   \lambda_{2} \left( r_{2} \frac{du}{u} + s_{2} \frac{dv}{v} + \frac{d \tilde{\phi}_{2}}{\tilde{\phi}_{2}} \right) +
   \lambda_{3} \left(r_{3} \frac{du}{u} + s_{3} \frac{dv}{v} + \frac{d \tilde{\phi}_{3}}{\tilde{\phi}_{3}} \right)
\right)  \\
 & = &
 r \frac{du}{u} + s \frac{dv}{v}
 + p_{1} \frac{d \tilde{\phi}_{1}}{\tilde{\phi}_{1}} + p_{2} \frac{d \tilde{\phi}_{2}}{\tilde{\phi}_{2}} +    \\
 &   &
 \varphi(u^{r}v^{s} \tilde{\phi}_{1}^{p_{1}} \tilde{\phi}_{2}^{p_{2}})
\left(
  \left( \lambda_{2}  r_{2} + \lambda_{3} r_{3} \right) \frac{du}{u} + \left( \lambda_{2} s_{2} + \lambda_{3} s_{3} \right)
  \frac{dv}{v} + \lambda_{2} \frac{d \tilde{\phi}_{2}}{\tilde{\phi}_{2}}  +
    \lambda_{3} \frac{d \tilde{\phi}_{3}}{\tilde{\phi}_{3}}
\right),
\end{eqnarray*}
where  $r = p_{1}r_{1} + p_{2}r_{2} $ and $s = p_{1}s_{1} + p_{2}s_{2} $.
Once more we have two subcases:

\par \noindent \underline{Subcase II.1}\ $\E$ has two branches, $\E =  \{u v = 0\}$.
Since
  $\phi$ takes $\E$ into $\DD$, we have
that either $r_{1}> 0$ or $r_{2} > 0$ and, by the same token,   either  $s_{1}> 0$ or $s_{2} > 0$.
Thus, both
 $r > 0$ and $s > 0$.
By cancelling poles, we find
\begin{equation}
\label{eq-phi*}
 uv \phi^{*} \tilde{\omega}   =   r v du + s u dv + uv \theta
 \end{equation}
for some formal $1-$form $\theta$. This
shows that $\G$ has a simple non-degenerate singularity, which is $\E$-well oriented.

\par \noindent \underline{Subcase II.2} \
$\E$ has a single branch, which we suppose to be $\E = \{v = 0\}$.
Since $\phi$    takes $\E$ into $\DD$, and then
either $s_{1} > 0$ or $s_{2} > 0$, which gives at once that $s>0$.
The curve $\Gamma = \{u = 0\}$ is mapped into one of the $\F$-invariant planes.
We have two possibilities:
\begin{itemize}[leftmargin=*]
\item If $\phi(\Gamma) \subset \DD$, then either $r_{1} > 0$ or $r_{2} > 0$ (not both),
implying that also
 $r > 0$. We are in the same situation of the preceding subcase:   $\G$ is induced by a $1-$form as in \eqref{eq-phi*},
 giving that it is simple non-degenerate and  $\E$-well oriented.
\item
If $\phi(\Gamma) \not\subset \DD$, then $\phi(\Gamma)  \subset\{z=0\}$. This implies that
  $r_{3} > 0$ and  $r_{1} =  r_{2} = 0$, giving $r =0$.
By cancelling poles, we get
\[uv \phi^{*} \tilde{\omega} = s  u dv +
 \lambda_{3} r_{3} v \varphi(v^{s} \tilde{\phi}_{1}^{p_{1}} \tilde{\phi}_{2}^{p_{2}}) du+
 u v  \, \theta ,\]
 for some formal $1-$form $\theta$. Thus, $\G$ has a saddle-node singularity
 at the origin whose weak separatrix is $\Gamma = \{u=0\}$. It is therefore    $\E$-well oriented.
\end{itemize}
\end{proof}

\begin{obss}
{\rm  We can replace, in   Lemma \ref{lemma-fund},  the hypothesis  ``$\G = \phi^{*} \F$ is $\E$-simple''  for ``$\G$  has exactly two transversal separatrices''. The fact of $\G$ being $\E$-simple is a consequence of
  the proof.
}\end{obss}

We can now extend to the three dimensional case  the     definition of \textit{second type foliation}.
As in dimension two, we take into account the final models in the process of   reduction of singularities:
\begin{ddef}
\label{second-type-3d}
{\rm
Let $\mathcal{F}$ be a germ of codimension one holomorphic foliation at $({\co}^3,0)$.   We say that  $\mathcal{F} $ is a \textit{second type} foliation  if  there exists a
reduction process $\pi: (\tilde{M},\DD) \to (\C^{3}, \sing(\F))$ such
that all singularities of $\tilde{\F} = \pi^{*} \F$
are $\DD$-well oriented.
}\end{ddef}

The definition of second type foliations is independent of the reduction of
singularities. Besides, it may be expressed in terms of two dimensional sections.
This is the content of the following result:

\begin{prop}
\label{teo-second-type-3d}
For a germ of   holomorphic foliation $\mathcal{F}$ at $(\mathbb{C}^3,0)$, the following facts
 are equivalent:
 \begin{enumerate}[label={(\arabic*)}]
\item   $\F$ is a second type foliation;
\item   for every normal crossings divisor $\E$ at $(\C^2,0)$
and every
analytic   map $\phi:(\C^2,\E)\to (\C^3,\sing(\F))$   generically transversal to $\mathcal{F}$
outside $\E$
  we have that
$\G = \phi^{*}\mathcal{F}$ is  $\E$-second type.
\item  for every
reduction of singularities $\pi: (\tilde{M},\DD) \to (\C^{3}, \sing(\F))$ for $\F$,
  all singularities of $\tilde{\F} = \pi^{*} \F$
are $\DD$-well oriented.
\end{enumerate}
\end{prop}
\begin{proof}
 \emph{(1)} $\Rightarrow$ \emph{(2)}.\
  Let $\pi:  (\tilde{M},\DD) \to (\mathbb{C}^3, \sing(\F))$ be a reduction of singularities for $\F$
such that all singularities of  $\tilde{\F} = \pi^{*} \F$ are $\DD$-well oriented and
  $\phi: (\C^{2},\E) \to (\mathbb{C}^3, \sing(\F))$ be an analytic map generically transversal to $\F$ outside
 a normal crossings divisor   $\E$. Take
  $\sigma: (\tilde{N}, \E^{\#}) \to (\C^{2},0)$   an $\E$-reduction of singularities and   denote
  $\tilde{\E} = \sigma^{-1}(\E) = \E^{\#} \cup \sigma^{*}\E$.
Then, by the universal property of blow-up maps, there exists an analytic map
$\psi: (\tilde{N}, \tilde{\E}) \to (\tilde{M}, \DD)$   such that the following diagram commutes:
 $$ \xymatrix{ \left(\tilde{N},\tilde{\E} \right) \ar[d]_{\sigma} \ar[r]^{\psi}
 \ar@{}[dr]|{\circlearrowright} &\left(\tilde{M},\mathcal{D}\right) \ar[d]^{\pi} \\ \left(\mathbb{C}^2, \E \right) \ar[r]^{\hspace{-0.5cm} \phi} & \left(\mathbb{C}^3,\sing(\F)\right) \, .} $$
 Note that
 \[ \tilde{\G} = \sigma^{*} \G =
 \sigma^{*} \phi^{*} \F =  \psi^{*} \pi^{*} \F = \psi^{*} \tilde{\F}.\]
 The result then follows by applying Lemma \ref{lemma-fund} to the local map
defined by $\psi$ at each simple singularity of $\tilde{\G}$ in order to conclude that
they are all $\tilde{\E}$-well oriented and, as a consequence, that $\G$ is $\E$-second type.

\emph{(2)} $\Rightarrow$ \emph{(3)}.\
Let $\pi:  (\tilde{M},\DD) \to (\mathbb{C}^3, \sing(\F))$ be a reduction of singularities for $\F$.
Suppose  that
  $\tilde{\F} = \pi^{*} \F$ has a tangent saddle-node, say at $p \in \DD$.
We can choose $p$  such that  $\tau_{p}(\tilde{\F}) = 2$.
Take $\rho: (\C^{2},0) \to (\tilde{M},p)$ an analytic embedding transversal to $\tilde{\F}$.
Setting $\E = \rho^{-1}(\DD)$, we have that $\G =\rho^{*} \tilde{\F}$ is an $\E$-tangent saddle-node.
Define $\phi = \pi \circ \rho :(\C^{2},\E) \to (\mathbb{C}^3, \sing(\F))$, making   the diagram 
\[
\xymatrix{ \ar@{}[dr]|{\hspace{1.7cm} \circlearrowright} &\left(\tilde{M},\mathcal{D}\right) \ar[d]^{\pi} \\
 \left(\mathbb{C}^2,\E\right)\ar[r]^{\hspace{-.5cm}\phi} \ar[ru]^{\rho}&
\left(\mathbb{C}^3,Sing(\mathcal{F})\right)}
\]
commute. We have that
$\G = \phi^{*} \F$ is an $\E$-tangent saddle-node singularity, which is in contradiction
with \emph{(2)}.

\emph{(3)} $\Rightarrow$ \emph{(1)}.\ This is evident, using the fact that a foliation
in ambient dimension three has a reduction of singularities.

\end{proof}

\begin{cor}
\label{cor-multiplicity}
Let $\F$ be a non-dicritical second type foliation at $(\C^{3},0)$ and
  $f= 0$, where $f \in \hat{\cl{O}}_{3}$, be a   reduced equation of   separatrices. Then $\nu_{0}(\F) = \nu_{0}(d f)$.
\end{cor}
\begin{proof} Let $\phi:(\C^{2},0) \to (\C^{3},0)$ be an embedding generically transversal to
$\F$. We can take $\phi$  satisfying  $\nu_{0}(\G) =  \nu_{0}(\F)$, where $\G = \phi^{*} \F$,
and $\nu_{0}(dg) =  \nu_{0}(df)$, where $g = \phi^{*}f = f \circ \phi$.
The foliation $\G$ is non-dicritical, second type   by Proposition \ref{teo-second-type-3d}, having $g=0$ as a reduced equation
of separatrices. Thus, the minimization
property of the algebraic multiplicity gives
$\nu_{0}(\G) = \nu_{0}(dg)$, proving the corollary.
\end{proof}



\section{Desingularization of separatrices}
\label{section-desingularization}

In this section, we prove that, in the non-dicritical case, a second type foliation and its set of
separatrices have the same reduction of singularities. The arguments are similar  to those in \cite{mozo2009}.

We begin by a  brief comment  on the method of Cano-Cerveau for the construction
of separatrices for a non-dicritical foliation $\F$ at $(\C^{3},0)$ \cite[Part IV]{cano1992}. Let
  $\tilde{\F} = \pi^{*}\F$, where $\pi:(\tilde{M},\DD) \to (\C^{3},\sing(\F))$
is a reduction of singularities for $\F$. Let $\cl{U} \subset \DD$ be the analytic set formed by  all
  trace singularities, that is, points $p \in \sing(\tilde{\F})$ such that
$\e_{p}(\DD) = \tau_{p}(\tilde{\F})-1$. Then, the blow-up map $\pi$ defines canonically  a bijection between $\sep_{0}(\F)$ and the
connected components of $\cl{U}$. This, in particular, has the following consequence:
if $\F$ is a non-dicritical foliation at $(\C^{3},0)$ and $\Gamma$ is an $\F$-invariant formal curve outside $\sing(\F)$, then
  $\Gamma \subset \sep_{0}(\F)$. Indeed, since $\DD$ is $\tilde{\F}$-invariant,    $\tilde{\Gamma} = \pi^{*} \Gamma$ touches $\DD$ at a   point $p$ that is singular for $\tilde{\F}$.
   By Lemma \ref{invariant-curve},
$\tilde{\Gamma}$ must be contained in a component of $\sep_{p}(\tilde{\F})$ which lies outside $\DD$. Thus,
  $\Gamma$ lies in the correspondent irreducible component of $\sep_{0}(\F)$.

We have the following analogue of   Lemma \ref{lemma-simple2} for foliations in dimension three:
\begin{lem}
\label{lemma-simple3}
Let $\F$ be a germ of non-dicritical foliation at $(\C^{3},0)$ and $\DD$ be an  $\F$-invariant    normal crossings divisor.
Suppose that $\sep_0(\F)$ is formed by $s_{0}(\F)= 2$ or $3$ smooth surfaces with normal crossings and  that
\begin{enumerate}[label={(\roman*)}]
\item  $s_{0}(\F)-1 \leq \e_{0}(\DD) \leq  s_{0}(\F)$;
\item  $\F$ is second type.
\end{enumerate}
Then $\F$ is   $\DD$-simple.
\end{lem}
\begin{proof}
An important point here is that $\sing(\F)$ is an analytic set, formed by $s_{0}(\F)!/2$ pairwise transversal smooth curves corresponding to the intersections of components of $\sep_0(\F)$.

Let us first suppose   $s_{0}(\F)= 2$.  Let $\phi: (\C^{2},0) \to (\C^{3},0)$ be a germ of analytic embedding
generically transversal to $\F$. It follows from Cano-Cerveau's method that $\G = \phi^{*} \F$ has exactly two invariant
curves, which are smooth and transversal. Besides, $\G$ is second type, by Proposition \ref{teo-second-type-3d},  and simple, by Lemma \ref{lemma-simple2}.
To see that $\F$ is also simple, it suffices to see it --- after a coordinate change --- as the unfolding
of the simple foliation $\G$ at $(\C^{2},0)$. This unfolding must be trivial by \cite[Lemma 1.1.5]{mattei1991}.
Thus, $\F$ has dimensional type $\tau = 2$ and is $\DD$-simple.

Suppose now  $s_{0}(\F)= 3$ and, once more, take
  a germ $\phi: (\C^{2},0) \to (\C^{3},0)$  of analytic embedding
generically transversal to $\F$, also satisfying $\nu_{0}(\F) = \nu_{0}(\G)$, where $\G = \phi^{*} \F$.
We have that $\G $ is a second type foliation with exactly three invariant
curves, smooth and pairwise transversal --- corresponding to the pre-images of the components
of $\sep_{0}(\F)$ --- which implies
 $ \nu_{0}(\G) = 2$. Choosing formal normalizing coordinates $(x,y,z)$
for $\sep_{0}(\F)$, we  find that
 $\F$ is induced by a $1-$form of the kind
\[ \omega = xyz \left( a \frac{dx}{x} + b\frac{dy}{y} + c \frac{dz}{z} \right) ,\]
where  $a,b,c \in \hat{\cl{O}}_{3}$ and, since   $\nu_{0}(\F) = 2$, at least one of them must be a unity. Then, it is a pre-simple singularity
(see \cite[Def. 2.2]{cano1992})
adapted to   $\DD$. On the other hand,    $\F$ has exactly two smooth transversal separatrices
at every point in
 $\sing(\F) \setminus \{0\}$ --- this too results from Cano-Cerveau's method.  Thus, by  the  case
$s_{0}(\F)= 2$, all these points are $\DD$-simple singularities.
We conclude  from \cite[Prop. 4.7 and Def. 4.8]{cano1992} that $\F$ is also $\DD$-simple at $0 \in \C^{3}$.
 \end{proof}

 We have all elements to prove the main result of this paper:

\begin{proof}(of Theorem \ref{theorem-main})
Evidently, a reduction of singularities for $\F$  desingularizes   $S = \sep_{0}(\F)$ and we   only have to work
 the inverse assertion. Let $\pi: (\tilde{M},\DD) \to (\C^{3},\sing(\F))$ be a composition of blow-ups
with non-singular centers that desingularizes $S$, that is, such that $\tilde{S} = \pi^{*}S$ is smooth
and has normal crossings with $\DD$. All such centers are permissible
for the reduction process of $\F$, since the foliation is non-dicritical
and $S$ is $\F$-invariant.
Now, the crucial fact is   that a point in the regular part of $\DD$ not lying in $\tilde{S}$ is regular for $\tilde{\F} = \pi^{*} \F$. Indeed, this follows from
 Corollary \ref{cor-multiplicity}, considering that,  at such a point, the unique local separatrix is contained in $\DD$.
 Therefore, all singularities of $\tilde{\F}$ satisfy the conditions of
  Lemma \ref{lemma-simple3}, leading to the conclusion that
they are all   $\DD$-simple.
\end{proof}


\section{Logarithmic foliations}
\label{section-logarithmic}

A singular holomorphic foliation $\F$ of degree $d \geq 0$  in $\pe^{3} = \pe^{3}_{\C}$ is given, in homogeneous coordinates
 $[X:Y:Z:W] \in \C^{4}$, by a polynomial $1-$form
  \[\omega =  A dX + Bdy + CdZ +D dW , \]
  where $A,B,C,D \in \C[X,Y,Z,W]$ are homogeneous of degree $d+1$ satisfying the Euler condition,
$A X + BY + CZ +D W = 0$, and the integrability condition, $\omega \wedge d \omega = 0$.
Let $S \subset \pe^{3}$ be a surface with reduced equation $F_1 \cdots   F_\ell =0$, where
$F_1,\ldots,  F_\ell \in \C[X,Y,Z,W]$ are irreducible homogeneous polynomials.
The foliation $\F$ is said to be \emph{logarithmic
with poles on} $S$ if it is defined by a   $1-$form of the kind
\[ \omega = F_1 \cdots   F_\ell \left(  \lambda_1    \frac{dF_1}{F_1} + \cdots +  \lambda_\ell    \frac{dF_\ell}{F_\ell}  \right) = F_1 \cdots   F_\ell \, \tilde{\omega} , \]
where
$\lambda_{1}, \ldots, \lambda_{\ell} \in \C^{*}$ comply with
\[   \lambda_{1} \deg F_1 + \cdots + \lambda_\ell \deg F_\ell = 0,\]
which is imposed by the Residue Theorem.
The set of poles of $\tilde{\omega}$ is  precisely  $S$, which is invariant by $\F$. Note that
\[ \deg \F =  \deg F_1 + \cdots + \deg F_\ell -2 = \deg S - 2 .\]
The foliation $\F$ may have some isolated singularities,
where it admits local holomorphic first integrals by Malgrange's Theorem \cite{malgrange1976}.
Let us then denote by $\sing_{2}(\F)$ the union of all components of codimension two
in $\sing(\F)$.
Some of these components are contained in $S$: the pairwise intersections
  of pole components
$\{F_{i} = F_{j} = 0\}$, as well as
codimension  two components of the singular set of each surface $\{F_{i} = 0\}$.
Outside  $S$, the $1-$form $\tilde{\omega}$ is holomorphic and closed, thus
each singularity has a local holomorphic first integral.
Besides, if $\F$ is non-dicritical,
at each point of  $\sing_{2}(\F) \cap S$, all local separatrices of $\F$ are contained in $S$.
Taking into account this description, we propose the following characterization of logarithmic foliations:

\begin{prop} Let $\F$ be a non-dicritical second type foliation in $\pe^{3}$. Suppose that there exists
 an algebraic surface $S \subset \pe^{3}$ invariant by $\F$ such that:
 \begin {enumerate}[label={(\roman*)}]
 \item for every $p \in \sing_{2}(\F) \cap S$,
 the local set of separatrices $\sep_{p}(\F)$ is contained in $S$;
 \item at every point outside $S$, the foliation $\F$ admits a holomorphic first integral.
 \end{enumerate}
 Then $\F$ is a logarithmic foliation with poles on $S$.
\end{prop}
\begin{proof} We apply the  two-dimensional version of this result proved in \cite{mol2017}. Let
 $i: \pe^{2} \to \pe^{3}$ be a linear embedding generically transversal to $\F$, set
 $\G = i^{*}\F$ and identify   $\pe^{2}$ and the plane $H = i(\pe^{2})$.  Such a map exists by \cite{camacho1992}.
 Denoting by $d_{0} = \deg (S)$ and $d = \deg(\F)$, we also have $d_{0} = \deg (S \cap H)$ and $d = \deg(\G)$.
  Denote by $p_{1},\ldots,p_{r} \in \sing(\G)$ the  points
of intersection between $H$ and the components of  $\sing_{2}(\F)$ in $S$. At the other points in $\sing(\G)$,  say $q_{1},\ldots,q_{s}$,
the foliation $\G$ has local  holomorphic first integrals.
Note that $S \cap H$  contains all local separatrices of $\G$  at the points $p_{1},\ldots,p_{r}$. Our calculations  is based on the following fact \cite[Theorem I]{mol2017}: if $\G$ is a non-dicritical second type foliation, then $\bb_{p}(\G) =  \cs_{p}(\G) + 2  \gsv_{p}(\G)$,
where $\bb$ is the Baum-Bott index,  $\cs$ and $\gsv$ are, respectively, the total --- that is, with respect to the complete set of
 separatrices ---  Camacho-Sad and   G\'omez-Mont-Seade-Verjovski indices (see \cite{brunella1997} for definitions and properties of indices). Using known formulas for the sum of the $\cs$ and $\gsv$-indices along an invariant algebraic curve, we find
 \[
\begin{array}{rcl}
\displaystyle \sum_{i=1}^{r} \bb_{p_{i}}(\G) & = &\displaystyle \sum_{i=1}^{r} \cs_{p_{i}}(\G) + 2  \gsv_{p_{i}}(\G) \medskip \\
& = & d_{0}^{2} + 2 \left((d + 2)d_{0} - d_{0}^{2}\right) \medskip \\
& =  & 2  (d + 2)d_{0} - d_{0}^{2}.
\end{array}
\]
On the other hand,
$\bb_{q_{j}}(\G) \leq 0$ for all $j = 1, \ldots, s$, since  local holomorphic first integrals exist  \cite{fernandez2017}. Thus

\[\sum_{i=1}^{r} \bb_{p_{i}}(\G) \geq  \sum_{p \in  \sing (\G)} \bb_{p}(\G) = (d+2)^{2}.\]
Comparing these two expressions, we find
$(d_{0} - (d + 2))^{2} \leq 0$,  which is possible if and only if $d_{0} = (d + 2)$.
This  implies, by \cite{brunella1997}, that  $\G$ is a logarithmic foliation, giving that also
  $\F$ is  logarithmic.

\end{proof}

\bibliographystyle{plain}
\bibliography{referencias}

\medskip \medskip
\medskip \medskip
\noindent
Gilberto D. Cuzzuol \\
Departamento de Matem\'atica \\
Universidade Federal de Itajub\'a \\
Rua Irmã Ivone Drumond, 200\\
35903-087
Itabira - MG \\
BRAZIL \\
gilcuzzuol@unifei.edu.br

\medskip \medskip
\medskip \medskip
\medskip \medskip

\noindent
Rog\'erio S. Mol \\
Departamento de Matem\'atica \\
Universidade Federal de Minas Gerais \\
Av. Ant\^onio Carlos, 6627
\ C.P. 702 \\
30123-970 -
Belo Horizonte - MG \\
BRAZIL \\
rsmol@mat.ufmg.br

\end{document}